\documentclass[10pt]{amsart}
\usepackage{graphicx,enumitem,dsfont}
\usepackage[colorlinks=false]{hyperref}
\usepackage{subfigure}
\usepackage{pdfsync}
\usepackage{dcolumn}
\usepackage{upref}

\newcommand{\norm}[1]{\left\Vert#1\right\Vert}
\newcommand{\order}[1]{\mathcal{O}\left(#1\right)}
\newcommand{\abs}[1]{\left|#1\right|}
\newcommand{\eps}{\varepsilon}

\newcommand{\Dt}{{\Delta t}}
\newcommand{\R}{\mathbb R}
\newcommand{\Z}{\mathbb{Z}}
\newcommand{\N}{\mathbb N}
\newcommand{\T}{\mathbb{T}}

\newcommand{\Oh}{\mathcal O}
\newcommand{\dott}{\, \cdot\,}
\newcommand{\seq}[1]{\left\{#1\right\}}

\newcommand{\Hil}{\mathbb{H}_{\rm per}}
\newcommand{\dH}{\mathbb H}
\newcommand{\dG}{\mathbb G}

\newcommand{\Dx}{{\Delta x}}

\newcommand{\Sm}{S^{-}}
\newcommand{\Sp}{S^{+}}
\newcommand{\Dp}{D_{+}}
\newcommand{\Dm}{D_{-}}
\newcommand{\Dpm}{D_{\pm}}
\newcommand{\Dtp}{D^t_{+}}
\newcommand{\ave}[1]{\tilde{#1}}

\DeclareMathOperator*{\Lip}{Lip}
\renewcommand{\Im}{{\rm Im}}

\newcommand{\test}{\varphi}

\newtheorem{theorem}{Theorem}[section]

\newtheorem{lemma}[theorem]{Lemma}

\newtheorem{remark}[theorem]{Remark}

\numberwithin{equation}{section}     
\numberwithin{figure}{section}
\numberwithin{table}{section}

\allowdisplaybreaks

\newcounter{asnr}

{\ifnum\value{asnr}=0 \stepcounter{asnr} 
  \begin{enumerate}[label=\textbf{A}.\arabic{enumi}]
    \else
    \begin{enumerate}[label=\textbf{A}.\arabic{enumi},resume] \fi}
{\end{enumerate}}

\newcounter{defnr}

{\ifnum\value{defnr}=0 \stepcounter{defnr} 
  \begin{enumerate}[label=\textbf{D}.\arabic{enumi}]
    \else
    \begin{enumerate}[label=\textbf{D}.\arabic{enumi},resume] \fi}
{\end{enumerate}}

\begin{document}
\title[FD scheme for BO~equation]{Convergence of Finite Difference schemes for \\ the Benjamin--Ono
  equation}

\author[Dutta]{R. Dutta} \address[Rajib Dutta]{\newline Department of Mathematics, University
  of Oslo, P.O.\ Box 1053, Blindern, NO--0316 Oslo, Norway}
\email[]{\href{rajib.dutta@cma.uio.no}{rajib.dutta@cma.uio.no}}

\author[Holden]{H. Holden} \address[Helge Holden]{\newline Department
  of Mathematical Sciences, Norwegian University of Science and
  Technology, NO--7491 Trondheim, Norway,\newline {\rm and} \newline
  Department of Mathematics,
  University of Oslo, P.O.\ Box 1053, Blindern, NO--0316 Oslo, Norway}
\email[]{\href{holden@math.ntnu.no}{holden@math.ntnu.no}}
\urladdr{\href{http://www.math.ntnu.no/~holden}{www.math.ntnu.no/\~{}holden}}

\author[Koley]{U. Koley} \address[Ujjwal Koley] {\newline Tata
  Institute of Fundamental Research Centre, Centre For
  Applicable Mathematics, \newline Post Bag No. 6503, GKVK Post
  Office, Sharada Nagar, Chikkabommasandra, \newline
  Bangalore-560065, India.}
\email[]{\href{ujjwal@math.tifrbng.res.in}{ujjwal@math.tifrbng.res.in}}

\author[Risebro]{N. H. Risebro} \address[Nils Henrik Risebro]{\newline
  Department of Mathematics,
  University of Oslo, P.O.\ Box 1053, Blindern, NO--0316 Oslo, Norway}
\email[]{\href{nilshr@math.uio.no}{nilshr@math.uio.no}}
\urladdr{\href{http://www.mn.uio.no/math/english/people/aca/nilshr/}%
  {http://www.mn.uio.no/math/english/people/aca/nilshr/}}

\date{\today}


\thanks{Supported in part by the Research Council of Norway and the
  Alexander von Humboldt Foundation.}

\begin{abstract}
In this paper, we analyze finite difference schemes for Benjamin--Ono
equation, $u_t= u u_x + H u_{xx}$, where $H$ denotes the Hilbert
transform. Both the decaying case on the full line
and the periodic case are considered. If the initial data are sufficiently regular, 
fully discrete finite difference schemes shown to converge to a classical solution.
Finally, the convergence is illustrated by several examples.
\end{abstract}

\maketitle


\section{Introduction}
\label{sec:intro}
This paper considers a fully discrete
finite difference scheme for the Benjamin--Ono (BO) equation.  The
BO equation models the evolution of weakly nonlinear
internal long waves. It has been derived by Benjamin \cite{benjamin}
and Ono \cite{ono} as an approximate model for long-crested
unidirectional waves at the interface of a two-layer system of
incompressible inviscid fluids, one being infinitely deep. In
non-dimensional variables, the initial value problem associated with
the BO equation reads
\begin{equation}
  \begin{cases}
    \label{eq:main}
    u_t= uu_x + H u_{xx}, \quad x\in\R, \ 0\le t \le T,& \\
    u|_{t=0}=u_0,&
  \end{cases}
\end{equation}
where $H$ denotes the Hilbert transform defined by
the principle value integral
\begin{equation*}
  H u(x) := \mathrm{P.V.} \, \frac{1}{\pi} \int_{\R} \frac{u(x-y)}{y} \,dy.
\end{equation*}
The BO equation is, at least formally, completely
integrable \cite{ablowitz} and thus possesses an infinite number of
conservation laws.  For example, the momentum and the energy,
given by
\begin{align*}
  M(u):= \int u^2 \,dx, \, \, \text{and} \, \, E(u):= \frac{1}{2} \int
  \abs{D_x^{1/2} u}^2 \,dx + \frac{1}{6} \int u^3 \,dx,
\end{align*}
are conserved quantities for solutions of \eqref{eq:main}.

We also consider the corresponding $2L$-periodic problem
\begin{equation}
  \begin{cases}
    \label{eq:main_per}
    u_t= uu_x + \Hil u_{xx}, \quad &x\in\T, \ 0\le t \le T, \\
    u|_{t=0}=u_0,& x\in\T 
  \end{cases}
\end{equation}
where $\mathbb{T}:= \R/2 L \Z$.  The periodic
Hilbert transform is defined by the principle value integral
\begin{align*}
  \Hil u(x) = \mathrm{P.V.} \frac{1}{2L} \int_{-L}^{L}
  \cot \Bigl(\frac{\pi}{2L} y\Bigr) u(x-y)\,dy.
\end{align*}
The initial value problem \eqref{eq:main} has been extensively studied
in recent years. Well-posedness of \eqref{eq:main} in $H^s(\R)$,
for $s > 3$ was proved by Iorio \cite{iorio}  using purely
hyperbolic energy methods.  Then, Ponce \cite{ponce} derived a local
smoothing effect associated to the dispersive part of the equation,
which combined with compactness methods, enabled him to prove
well-posedness also for  $s = 3$.

By combining a complex version of the Cole--Hopf transform with
Strichartz estimates, Tao \cite{Tao:2004} was able to show
well-posedness of the Cauchy problem \eqref{eq:main} in $H^1(\R)$.
This well-posedness was extended to $H^s(\R)$ for $s>1$ by Burq and
Planchon \cite{burq} and for $s\ge 0$ by Ionescu and Kenig
\cite{ionescu}.
In the periodic setting, Molinet \cite{molinet3a}
proved well-posedness in $H^s (\mathbb{T})$ for  $s \ge
0$. For operator splitting methods applied to the BO equation, see \cite{DuttaHoldenKoleyRisebro:2015}.

In this paper, we define a numerical scheme for both \eqref{eq:main}
and \eqref{eq:main_per}, with the aim to develop a \textit{convergent}
finite difference scheme.  While there are several numerical methods
for the BO equation
which perform well in practice, indeed better than the one presented here, see \cite{BoydXu} for a recent comparison of different
numerical methods, we emphasize that we here \textit{prove} the convergence of our
proposed scheme. Having said this, there are results concerning error estimates 
for the BO equation  in \cite{vasu:1998, Pelloni:2001, DengMa:2009}.
However, error estimate analysis a priori assumes existence of solutions of the underlying equation, while
our convergence analysis, as a by-product, can be viewed as a constructive proof for the existence
of solutions of the BO equation \eqref{eq:main}.
It is worth mentioning that the scheme under consideration in this paper is similar to the scheme
analyzed in \cite{vasu:1998}, the only difference being that a different discretization of Hilbert transform is
introduced in this paper.


We analyze the  fully discrete Crank--Nicolson difference scheme 
\begin{equation}
  \label{eq:hilbertrealline}
  u^{n+1}_j = u^n_j + \Dt\, \ave{u}^{n+1/2}_j D u^{n+1/2}_j + \Dt\, \dH\left(
    \Dp \Dm u^{n+1/2} \right)_j,  
  \quad n\in\N_0, \, j\in \Z,
\end{equation}
where $\Dx, \Dt$ are discretization parameters, $u^n_j\approx u(j\Dx, n\Dt)$ and $u^{n+1/2}=(u^n+u^{n+1})/2$.  Furthermore, $D$ and $\Dpm$ denote
symmetric and forward/backward (spatial) finite differences,
respectively, $\dH$ denotes a discrete Hilbert transform operator, and $\ave{u}$ denotes a spatial average.  We show
(Theorem~\ref{thm:H3convergence}) that for initial data $u_0\in
H^2(\R)$ there exists a finite time $T$, depending only on the
$H^2(\R)$ norm of the initial data such that for $t\le T$, the
difference approximation \eqref{eq:BO_fd} converges uniformly in
$C(\R\times [0,T])$ to the unique solution of the BO
equation \eqref{eq:main} as $\Dx\to 0$ with $\Dt=\Oh(\Dx)$. 
Furthermore, following \cite[Theorem 3.2]{vasu:1998}, a second-order
error estimate in both time and space for smooth solutions can be obtained by our numerical method.
  
The rest of the paper is organized as follows: In
Section~\ref{sec:scheme}, we present necessary notations to introduce
the Crank--Nicolson
 scheme and present the convergence analysis in the full line case, in
 Section~\ref{sec:periodic} we present the periodic Hilbert transform
 and outline the proofs in the periodic setting, and finally in
Section~\ref{sec:numerics}, we test our numerical scheme and provide
some numerical results.

\section{The finite difference scheme}
\label{sec:scheme}
Throughout this paper, we use the letters $C$, $K$ etc.~to denote
various constants which may change from line to line. 
We start by introducing the necessary notation.  Derivatives will be
approximated by finite differences, and the basic quantities are as
follows.  For any function $p:\R\to \R$, we set
$$
D_{\pm} p(x)=\pm \frac1{\Dx}\big(p(x\pm \Dx)-p(x)\big), \ \text{and}\
D = \frac{1}{2}\left(\Dp + \Dm \right)
$$
for some (small) positive number $\Dx$. If we introduce the averages
$$
\ave{p}(x) := \frac{1}{3}\left(p(x+\Dx)+p(x)+ p(x-\Dx)\right), \,\, \bar{p}(x) := \frac{1}{2}\left(p(x+\Dx)+ p(x-\Dx)\right)
$$
and the shift operator
$$
S^\pm p(x)=p(x\pm\Dx),
$$
we find that
\begin{align*}
  D(pq)&= \bar{p} Dq+\bar{q} Dp, \\
  \Dpm (pq)&=S^\pm p \Dpm q+q\Dpm p=S^\pm q \Dpm p+p\Dpm q.
\end{align*}

We discretize the real axis using $\Dx$ and set $x_j = j \Dx$ for $j
\in \Z$. For a given function $p$, we define $p_j = p(x_j)$.  We will
consider functions in $\ell^2$ with the usual inner product and norm
\begin{equation*}
  \langle p, q\rangle= \Dx\sum_{j\in \Z} p_j q_j, \quad \norm{p}= \norm{p}_2=\langle p,p\rangle^{1/2}, \qquad
  p,q\in \ell^2. 
\end{equation*}  
Moreover, we define $h^2$-norm of a grid function as
\begin{align*}
\norm{p}_{h^2} := \Bigl(\norm{p}^2 + \norm{\Dp p}^2 + \norm{\Dp \Dm p}^2\Bigr)^{1/2}.
\end{align*} 
 Observe that
\begin{equation*}
  \norm{p}_\infty := \sup_{j\in\Z}\abs{p_j}\le \frac1{\Dx^{1/2}  } \norm{p}.
\end{equation*} 
In the periodic case, let $N$ be a given \emph{odd} natural number.
We divide the periodicity interval $[-L,L]$ into $N$ sub-intervals
$[x_j, x_{j+1}]$ using $ \Dx=\frac{2L}{N}$, where
\begin{equation*}
  x_j=-L+j\Dx, \quad \text{for} \quad j=0,1,2,.....,N.
\end{equation*}
In the periodic case the sum over $\Z$ is replaced by a
finite sum $j=0,\dots,N$. 
The various difference operators enjoy the following properties:
\begin{equation*}
  \langle p, D_{\pm}q\rangle=-\langle D_{\mp}p,q\rangle, \quad \langle p,Dq\rangle=-\langle Dp,q\rangle, \qquad p,q\in \ell^2.
\end{equation*}  
Furthermore, using Leibniz rules, the following identities can be readily verified:
\begin{subequations}
\begin{align}
\label{imp1}
\langle D(pq), q \rangle &= \frac{\Dx}{2} \langle \Dp p \,D q, q \rangle + \frac12 \langle S^{-} q\, Dp, q \rangle, \\
  \label{imp2}
 \Dp \Dm (pq)&=\Dm p \Dp q + \Dp \Dm q  S^- p + \Dp p \Dp q +  q \Dp \Dm p.
\end{align}
\end{subequations}

We also need to discretize in the time direction. Introduce (a small)
time step $\Dt>0$, and use the notation
\begin{equation*}
  \Dtp p(t)=\frac1{\Dt}\big(p(t+\Dt)-p(t)\big),
\end{equation*}
for any function $p\colon[0,T]\to \R$. Write $t_n=n\Dt$ for
$n\in\N_0=\N\cup\{0\}$. A fully discrete grid function is a function
$u_\Dx\colon \Dt\, \N_0 \to \R^\Z$, and we write
$u_\Dx(x_j,t_n)=u^n_j$. (A CFL-condition will enforce a relationship
between $\Dx$ and $\Dt$, and hence we only use $\Dx$ in the notation.)

Next we present a lemma, which essentially gives
a relation between discrete and continuous Sobolev norms. Since we shall use this lemma 
frequently, for the sake of completeness, we present a proof of this lemma in the full line case.
\begin{lemma}
  \label{lemmaimp}
  There exists a constant $C$ such that for all $u \in H^2(\R)$
  \begin{align*}
    \norm{u_\Dx}_{h^2} \le C\, \norm{u}_{H^2},
  \end{align*}
  where we identify $u_\Dx$ with the discrete evaluation
  $\seq{u(x_j)}_{j}$.
\end{lemma}
\begin{proof}
To begin with, observe that the discrete operator
$\Dp \Dm$ commutes with the continuous operator $\partial_x$. A simple use of the H\"{o}lder estimate reveals that
\begin{align*}
\norm{\Dp \Dm u}^2_{L^2(\R)}  &= \Dx \sum_j  \left(\frac{1}{\Dx} \left(\Dm u(x_{j+1}) - \Dm u(x_j)\right)\right)^2 \\
&= \Dx \sum_j \left(\int_{x_j}^{x_{j+1}} \frac{1}{\Dx} \partial_x \Dm u(x)\,dx\right)^2 \\
&\le \Dx \sum_j  \left(\norm{\frac{1}{\Dx}}_{L^2([x_j,x_{j+1}])} \norm{\partial_x \Dm u(x)}_{L^2([x_j,x_{j+1}])} \right)^2 \\
&= \norm{\Dm \partial_x u}^2_{L^2(\R)}.
\end{align*}
Similarly, we can show that
\begin{align*}
 \norm{\Dm \partial_x u}_{L^2(\R)} \le  \norm{\partial^2_x u}_{L^2(\R)}.
\end{align*} 
Furthermore, similar arguments can be used to show
\begin{align*}
\norm{\Dp u}_{L^2(\R)}  \le  \norm{\partial^2_x u}_{L^2(\R)}, \,\, \text{and} \,\, 
\norm{u}_{L^2(\R)}  \le  \norm{\partial^2_x u}_{L^2(\R)}.
\end{align*}
Combining above results, the result is proved.
\end{proof}

We will now provide details for the discrete Hilbert transform,
which is different in full line and the periodic cases. 

Here we concentrate on the full line case, both regarding the Hilbert
transform and the difference scheme. The periodic case is similar, and
we will only provide detailed proofs where the differences are
sufficiently important. Thus for the moment, we consider the
non-periodic case, while the results in the periodic case are outlined
in Section~\ref{sec:periodic}.

\subsection*{The discrete Hilbert transform on $\R$}
Recall that the continuous Hilbert transform $H$ on $\R$ is defined by
\begin{equation}\label{eq:HIlbert}
\begin{aligned}
  H(u)(x) &= \mathrm{P.V.} \, \frac{1}{\pi} \int_{\R} \frac{u(y)}{x-y}
  \,dy\\
  &=\lim_{\eps\downarrow 0} \frac{1}{\pi}\int_\eps^\infty
  \frac{1}{y}\left(u(x-y)-u(x+y)\right)\,dy.
\end{aligned}
\end{equation}
As a strategy to discretize the continuous Hilbert transform, we first consider \emph{even} $j$, and write $(H u)(x_j) := H(u)_j$
as
\begin{align*}
H(u)_j =\mathrm{P.V.} \frac{1}{\pi} \int_{\R} \frac{u(y)}{x_j-y} \,dy.
\end{align*}
This can be rewritten as
\begin{align*}
H(u)_j=\frac{1}{\pi}\sum_{k=\,\text{even}} \int_{x_k}^{x_{k+2}}
  \frac{u(y)}{x_j-y} \,dy.
\end{align*}
Next, we apply the midpoint rule on each of these integrals in the sum, to
obtain the following quadrature formula
\begin{align*}
H(u)_j\approx \frac{2}{\pi}\sum_{k =\,\text{odd}} \frac{u_k}{j-k} .
\end{align*}
Similar arguments can be repeated almost \emph{verbatim} to deal with \emph{odd} $j$, to conclude 
\begin{equation*}
H(u)_j\approx \frac{2}{\pi}\sum_{k=\,\text{even}} \frac{u_k}{j-k}.
\end{equation*}
Therefore, combining the above results, we can define the discrete
Hilbert transform $\dH$ of a function $u$ as
\begin{align}
  \label{eq:hil_full}
  \dH(u_\Dx)_j&=\frac1{\pi} \sum_{k\ne j}\frac{u_k
    \left(1-(-1)^{j-k}\right)}{j-k}\qquad j\in \Z\\
  &=\frac{1}{\pi} \sum_{k=1}^\infty \frac{1}{k}\left(u_{j-k} -
    u_{j+k}\right)\left(1-(-1)^k\right)\notag\\
  &=\frac{1}{\pi} \sum_{k=0}^\infty \int_{x_{2k}}^{x_{2k+2}} 
  \frac{1}{x_{2k+1}} \left(u(x_j-x_{2k+1})-u(x_j+x_{2k+1})\right)\,dy.\notag
\end{align}
We now list some useful properties of \eqref{eq:hil_full} in the
following lemma.
\begin{lemma}
  \label{lemma:imp}
  The discrete Hilbert transform $\dH$ on $\R$ defined by
  \eqref{eq:hil_full} is a linear operator with the following
  properties: \\
(i)  (Skew symmetric) For any two grid functions $u$ and $v$, the
    discrete Hilbert transform satisfies
    \begin{align*}
      \langle \dH u, v \rangle = - \langle u, \dH v \rangle.
    \end{align*}
(ii) (Translation invariant) The discrete Hilbert transform
    commutes with discrete derivatives, i.e., 
    \begin{align*}
    \dH\left( \Dpm u \right) = \Dpm \dH(u).
    \end{align*}
 (iii) (Norm preservation) Finally, it also preserves the discrete
    $L^2$-norm
    \begin{align*}
      \norm{\dH u}=\norm{u}.
    \end{align*}
\end{lemma}
\begin{remark}
The continuous Hilbert transform \eqref{eq:HIlbert} satisfies the same properties with respect to the standard inner product in $L^2$ and ordinary derivatives.
\end{remark}
For a proof of the above lemma, we refer to the monograph by King \cite[pp.~671--674]{king}. 
It is worth mentioning that these properties are
essential in order to carry out the analysis given below.
We shall also have use for the following lemma:
\begin{lemma}
  \label{lem:l2convergence} Let $\test$ be a function in $C^3_0(\R)$,
  and define the piecewise constant function $h_{\Dx}$ by
  \begin{equation*}
    h_{\Dx}(x)=h_j=\dH(\test)(x_j)\quad \text{ for $x\in [x_j,x_{j+1})$.}
  \end{equation*}
  Then 
  \begin{equation*}
    \lim_{\Dx\to 0} \norm{H(\test) - h_{\Dx}}_{L^2(\R)} = 0.
  \end{equation*}
\end{lemma}
\begin{proof}
  We define the auxiliary function 
  \begin{equation*}
    \tilde{h}(x)=\tilde{h}_j=H(\test)(x_j)\quad \text{ for $x\in [x_j,x_{j+1})$.}
  \end{equation*}
  Then
  \begin{align*}
    \norm{H(\test) - \tilde{h}}_{L^2(\R)}^2 &=
    \sum_j \int_{x_j}^{x_{j+1}} \left(H(\test)(x) -
      H(\test)(x_j)\right)^2 \,dx \\
    &= \sum_j \int_{x_j}^{x_{j+1}} \Bigl( \int_{x_j}^x
    H(\test)'(z)\,dz\Bigr)^2 \,dx\\
    &= \sum_j \int_{x_j}^{x_{j+1}} \Bigl( \int_{x_j}^x
    H(\test')(z)\,dz\Bigr)^2 \,dx\\
    &\le \sum_j \int_{x_j}^{x_{j+1}} \int_{x_j}^{x_{j+1}} 
    (H(\test')(z))^2\,dz\, (x-x_j) \,dx \\
    &= \frac{\Dx^2}{2} \norm{H(\test')}_{L^2(\R)}^2 \\
    &= \frac{\Dx^2}{2} \norm{\test'}_{L^2(\R)}^2 .
  \end{align*}
  Next,
  \begin{align*}
    \norm{h-\tilde{h}}_{L^2(\R)}^2 &= \Dx\sum_{j} (h_j-\tilde{h}_j)^2
    \le \Dx\sum_{\abs{j}\le J} (h_j-\tilde{h}_j)^2 +
    2\Dx\sum_{\abs{j}>J} h_j^2 + \tilde{h}_j^2\\
    &=:S_1+S_2.
  \end{align*}
  Now we have that
  \begin{align*}
    h_j-\tilde{h}_j &=\sum_{k\ge 0}\Big( \int_{x_{2k}}^{x_{2k+2}}
    \psi(x_j,x_{2k+1}) \,dy - \int_{x_{2k}}^{x_{2k+2}} \psi(x_j,y)\,dy\Big),
  \end{align*}
  where $\psi(x,y)=(\test(x-y)-\test(x+y))/y$. By the error formula
  for the midpoint quadrature rule we have that 
  \begin{equation*}
   \Bigl|  \int_{x_{2k}}^{x_{2k+2}}
    \psi(x_j,x_{2k+1}) \,dy - \int_{x_{2k}}^{x_{2k+2}} \psi(x_j,y)\,dy \Bigr| \le C\Dx^3 \norm{\test^{(3)}}_{L^\infty(\R)}.
  \end{equation*}
  Furthermore, since the support of $\test$ is bounded, the above sum
  over $k$ contains only a finite number of terms, namely $M_\test/\Dx$,
  independently of $j$.
  Therefore,
  \begin{equation*}
    \abs{h_j-\tilde{h}_j}\le M_\test  C \Dx^2 \norm{\test^{(3)}}_{L^\infty(\R)},
  \end{equation*}
  and
  \begin{equation*}
    S_1 \le M_\test  C \Dx^4 \norm{\test^{(3)}}_{L^\infty(\R)} 2 J.
  \end{equation*}
  Since $\sum_j h_j^2$ and $\sum_j \tilde{h}_j^2$ are finite, we can
  choose $J$ large to make $S_2$ small, and then $\Dx$ small to make
  $S_1$ small. Hence $\norm{h-\tilde{h}}_{L^2}$ converges to zero as
  $\Dx \to 0$. By the triangle inequality $\norm{H(\test)-h}_{L^2}\le
  \norm{H(\test)-\tilde{h}}_{L^2}+\norm{h-\tilde{h}}_{L^2} \to 0$.
\end{proof}

\subsection*{The difference scheme}
We propose the following Crank--Nicolson implicit scheme to generate
approximate solutions of the BO equation \eqref{eq:main}
\begin{equation}
  \label{eq:BO_fd}
  u^{n+1}_j = u^n_j + \Dt\, \dG(u^{n+1/2})_j + \Dt\, \dH( \Dp \Dm u^{n+1/2})_j, 
  \hspace{.3cm} n\in\N_0,\ j\in \Z,
\end{equation}
where we have used the following notations:
\begin{align*}
u^{n+1/2} := \frac12(u^n + u^{n+1}), \quad \text{and} \,\, \dG(u):= \ave{u}\,Du.
\end{align*}
For the initial data we have
\begin{equation*}
  u^0_j=u_0(x_j),\hspace{.3cm} j\in \Z.
\end{equation*}
Note that since the scheme \eqref{eq:BO_fd} is implicit, we must guarantee that
the scheme is well-defined, i.e., that it admits a unique solution.  Assuming this for the moment, we show 
that the implicit scheme is $L^2$-conservative, by simply taking
inner product of the scheme \eqref{eq:BO_fd} with $u^{n+1/2}_j$. This yields 
\begin{equation*}
  \frac12  \langle u^{n+1}-u^n, u^{n+1} +u^n \rangle = \Dt  \langle u^{n+1/2}, \dG u^{n+1/2}\rangle 
  + \Dt \langle u^{n+1/2}, \dH(\Dp \Dm u^{n+1/2})\rangle.
  \end{equation*}
A simple calculation, using Lemma~\ref{lemma:imp}, reveals that 
\begin{equation}\label{eq:simple}
\langle \dH \left(\Dp \Dm u \right), u \rangle=0, \quad \text{and}\,\, \langle \dG(u), u \rangle=0.
\end{equation} 
Thus, we conclude that
\begin{equation}\label{eq:l2}
  \norm{u^{n+1}}  = \norm{u^{n}} . 
\end{equation}
To solve \eqref{eq:BO_fd}, we use a simple fixed point iteration, and define the
sequence $\seq{w_{\ell}}_{\ell\ge 0}$ by letting $w_{\ell+1}$ be the
solution of the linear equation
\begin{equation}
\label{eq:iteration scheme}
  \begin{cases}
   w_{l+1} =v + \Dt \,\dG\left(\frac{v+w_l}{2}\right) + \frac{1}{2}\Dt \,\dH \Big( \Dp \Dm \left(v+w_{l+1}\right)\Big),  \\
    w^0 = v := u^n.
  \end{cases}
\end{equation}
See also \cite[Lemmas 3.3 and 3.5]{vasu:1998}.

The following stability lemma serves as a building block for the subsequent convergence analysis. 
\begin{lemma}
  \label{lemma1}
Choose a constant $L$ such that $0<L<1$ and set
  \begin{equation*}
    K=\frac{6-L}{1-L}>6.
  \end{equation*}
We consider the iteration \eqref{eq:iteration
    scheme} with $w^0=u^n$, and assume that the following CFL
  condition holds
  \begin{equation}\label{eq:cfl}
    \lambda \le L/\left(K \norm{u^n}_{h^2}\right), \,\, \text{with}\,\,  \lambda = \Dt /\Dx.
  \end{equation}
Then there exists a function $u^{n+1}$ which solves
\eqref{eq:BO_fd}, and $\lim_{\ell\to\infty}w^\ell = u^{n+1}$. Furthermore,  
the following estimate holds:
  \begin{equation}
    \label{eq:unp1bnd}
    \norm{u^{n+1}}_{h^2}\le K \norm{u^n}_{h^2},
  \end{equation}
  where $K$ depends only on given $L$.
\end{lemma} 
\begin{proof}
Define  $\Delta w_l := w_{l+1}-w_l$, a straightforward calculation using \eqref{eq:iteration scheme} returns
\begin{equation}
\label{eq:test1}
  \Big(1-\frac{1}{2}\Dt\,\dH \Dp \Dm \Big) \Delta w_l = \Dt
  \left[\dG\left(\frac{v+w_l}2\right) -
    \dG\left(\frac{v+w_{l-1}}2 \right)\right]=:\Dt \Delta \dG. 
\end{equation}
Next, applying the discrete operator $\Dp \Dm$ to \eqref{eq:test1}, then multiplying the resulting equation by $\Dx \Dp \Dm \Delta w_l$,
and subsequently summing over $j \in \Z$, we conclude
\begin{equation*}
  \norm{\Dp \Dm \Delta w_l}^2 = \Dt\langle \Dp \Dm \Delta \dG,\Dp \Dm \Delta w_{l}\rangle \le \Dt
  \norm{\Dp \Dm \Delta \dG}\,\norm{\Dp \Dm \Delta w_l}.
\end{equation*}
After some calculations, we find that
\begin{equation*}
  \Delta \dG = \frac14\left[
    \widetilde{\Delta w_{l-1}} D\left(v+w_{l-1}\right) +
    \widetilde{(v+w_l)} D\left(\Delta w_{l-1}\right)\right].
\end{equation*}
Next, in order to calculate $\Dp \Dm \Delta \dG$, we use the identity \eqref{imp2} and discrete
Sobolev inequalities (cf. \cite[Lemma A.1]{HoldenKoleyRisebro:2013}). This results in
\begin{align*}
\norm{\Dp \Dm  \left(\widetilde{\Delta w_{l-1}}  D\left(v+w_{l-1}\right) \right)} 
 \le \frac{1}{\Dx} \max{\Big\{\norm{v}_{h^2}, \norm{w_{l-1}}_{h^2}\Big\}}
\norm{\Delta w_{l-1}}_{h^2},
\end{align*}
and similarly
\begin{align*}
\norm{\Dp \Dm \left( \widetilde{(v+w_l)} D\left(\Delta w_{l-1}\right)\right)} 
 \le \frac{1}{\Dx}\max{\Big\{\norm{ v}_{h^2},\norm{w_l}_{h^2}\Big\}} \norm{\Delta w_{l-1}}_{h^2}.
\end{align*}
Combining the above results, we obtain
\begin{equation}
\label{eq:ineq}
  \norm{\Dp \Dm \Delta w_l}\le \lambda
 \max{\Big\{\norm{ v}_{h^2},\norm{w_l}_{h^2}, \norm{w_{l-1}}_{h^2} \Big\}} \norm{\Delta w_{l-1}}_{h^2}.
\end{equation}
Observe that an appropriate inequality like \eqref{eq:ineq} can be obtained for $\norm{\Dp \Delta w_l}$ and $\norm{\Delta w_l}$,
which in turn can be used, along with \eqref{eq:ineq}, to conclude
$$
 \norm{\Delta w_l}_{h^2}\le \lambda
 \max{\Big\{\norm{ v}_{h^2},\norm{w_l}_{h^2}, \norm{w_{l-1}}_{h^2} \Big\}} \norm{\Delta w_{l-1}}_{h^2}. 
$$
To proceed further, we need to estimate $\norm{\Dp \Dm w_l}$. In that context, we first observe that $w_1$
satisfies the following equation
\begin{equation*}
  w_1=v + \Dt \,\dG(v) + \frac{1}{2}\Dt\,\dH \Big( \Dp \Dm (v+w_1)\Big).
\end{equation*}
Applying the discrete operator $\Dp \Dm$ to the equation satisfied by $w_1$,
and subsequently taking the inner product with $ \Dp \Dm(v+w_1)$, we get
\begin{align*}
   \norm{\Dp \Dm w_1}^2 &=  \norm{\Dp \Dm v}^2 + \Dt
  \langle\Dp \Dm \dG(v),\Dp \Dm(v+w_1)\rangle) \\
  &= \norm{\Dp \Dm v}^2 +  \Dt \langle(\Dp \Dm \dG(v),\Dp \Dm w_1\rangle\\
  &\le \norm{\Dp \Dm v}^2 + \Dt^2 \norm{\Dp \Dm \dG(v)}^2 + \frac14 \norm{\Dp \Dm w_1}^2.
\end{align*}
Next, a simple calculation along with
discrete Sobolev inequalities (cf. \cite[Lemma A.1]{HoldenKoleyRisebro:2013}) confirms that 
\begin{equation*}
  \norm{\Dp \Dm \dG(v)} = \norm{\Dp \Dm (\tilde{v}\,Dv)} \le \frac{2}{\Dx} \norm{v}_{h^2}^2.
\end{equation*}
Hence
\begin{equation}
  \norm{\Dp \Dm w_1}\le \sqrt{\frac43} \left(1 + 4\lambda^2 \norm{v}_{h^2}^2\right)^{1/2}\norm{v}_{h^2}. \label{eq:ineq1}
\end{equation}
Now choose a constant $L\in (0,1)$, and define $K$ by
\begin{equation*}
  K=\frac{6-L}{1-L}>6.
\end{equation*}
Therefore, it is clear that if $\lambda$ satisfies the CFL condition \eqref{eq:cfl}, then
\begin{equation*}\label{eq:lambdaassume1}
  \sqrt{\frac43}\sqrt{1 + 4\lambda^2 \norm{v}_{h^2}^2}\le 4.
\end{equation*}
Hence from \eqref{eq:ineq1}, making use of the interpolation inequality, we conclude that
\begin{align*}
 \norm{w_1}_{h^2} &\le K \norm{v}_{h^2}.
\end{align*}
At this point, we assume inductively that
\begin{subequations}
  \begin{align}
    \label{eq:inducta}
    \norm{w_l}_{h^2} &\le K \norm{ v}_{h^2}, \quad \text{for}\,\, l=1,\ldots,m,\\
    \label{eq:inductb}
    \norm{ \Delta w_l}_{h^2}&\le L \norm{\Delta w_{l-1}}_{h^2}, \quad \text{for}\,\, l=2,\ldots,m.
  \end{align}
\end{subequations}
We have already shown \eqref{eq:inducta} for $m=1$. To show
\eqref{eq:inductb} for $m=2$, note that
\begin{align*}
  \norm{\Delta w_2}_{h^2} \le \lambda \max\Big\{\norm{v}_{h^2},\norm{ w_1}_{h^2}\Big\}
  \norm{ \Delta w_1}_{h^2}
  \le 4 \lambda \norm{ v}_{h^2} \norm{ \Delta w_1}_{h^2}
  \le L \norm{ \Delta w_1}_{h^2},
\end{align*}
by CFL condition \eqref{eq:cfl}. 
To show \eqref{eq:inducta} for $m>1$, 
\begin{align*}
 \norm{ w_{m+1}}_{h^2} & \le \sum_{l=0}^m \norm{ \Delta w_l}_{h^2} + \norm{ v}_{h^2} 
  \le \norm{(w_1-v)}_{h^2} \sum_{l=0}^m L^l + \norm{ v}_{h^2} \\
  &\le \left(\norm{w_1}_{h^2} +\norm{ v}_{h^2} \right) \frac{1}{1-L} + \norm{ v}_{h^2} 
  \le  \frac{4+2-L}{1-L} \norm{v}_{h^2} =K\norm{ v}_{h^2}.
\end{align*}
Then 
\begin{equation*}
  \norm{\Delta w_{m+1}}_{h^2} \le \lambda K \norm{ v}_{h^2} \,\norm{ \Delta w_m}_{h^2} \le L
  \norm{\Delta w_m}_{h^2},
\end{equation*}
if the CFL condition \eqref{eq:cfl} holds.

To sum up, if $L\in (0,1)$, and $K$ is defined by $K=(6-L)/(1-L)$, and
$\lambda$ satisfies the CFL-condition 
\begin{equation*}
  \lambda \le \frac{L}{K\norm{ v}_{h^2} },
\end{equation*}
then we have the desired estimate ￼\eqref{eq:unp1bnd}. Finally,
using \eqref{eq:inductb}, one can show that
$\{w_{\ell}\}$ is Cauchy, hence $\{w_{\ell}\}$ converges. This
completes the proof.
\end{proof}

\begin{remark}
Observe that the above result guarantees that the iteration
scheme converges for one time step under CFL condition \eqref{eq:cfl},
where the ratio between temporal and spatial mesh sizes
must be smaller than an upper bound that depends on the computed
solution at that time, i.e., $u^n$. 
Since we want
the CFL-condition only to  depend on the initial data
$u_0$, we have to derive local a priori bounds for the computed
solution $u^n$. This will be achieved in Theorem~\ref{thm2} to conclude that the
iteration scheme \eqref{eq:iteration scheme} converges for
sufficiently small $\Dt$.
\end{remark}

The following lemma is the most important step towards stability, and the very heart of this paper:
\begin{lemma}
  \label{lemma3.1}
  Let the approximate solution $u^n$ be generated by the
  Crank--Nicolson scheme \eqref{eq:BO_fd}, where $\Dt$ and $\Dx$ are
  such that \eqref{eq:cfl} holds.
  Then  we have that  
  \begin{align*}
    D_+^t\left( \norm{ u^n}_{h^2}\right) \le \sqrt{\frac32} \,
      \norm{ u^{n+1/2}}_{h^2}^2.
  \end{align*}
\end{lemma}
\begin{proof}
If $\Dp\Dm u^n=0$, then $u^n=0$ and $u^{n+1}=0$ since $u^n, u^{n+1}\in  \ell^2$, so that the lemma
trivially holds. Therefore we can assume that $\Dp\Dm u^n\ne 0$.

Applying the discrete operator $\Dp \Dm$ to \eqref{eq:BO_fd}, and subsequently taking inner product with $\Dp \Dm u^{n+1/2}$ yields
\begin{align*}
\frac12 \norm{\Dp \Dm u^{n+1}}^2 &= \frac12 \norm{\Dp \Dm u^{n}}^2 + \Delta t \langle \Dp \Dm \dG(u^{n+1/2}), \Dp \Dm u^{n+1/2} \rangle,
\end{align*}
using \eqref{eq:simple}, which implies
\begin{equation}\label{eq:dtbound}
D_+^t\left(\norm{\Dp \Dm u^n}\right) = 2 \frac{ \langle \Dp \Dm \dG(u^{n+1/2}), \Dp \Dm u^{n+1/2} \rangle}{\norm{\Dp \Dm u^{n+1}} + \norm{\Dp \Dm u^n}}.
\end{equation}
For the moment we drop the superscript $n+1/2$ from our notation, and use the notation $u$ for $u^{n+1/2}$,
where $n$ is fixed. We use the product rule \eqref{imp2} to write
\begin{align*}
  \langle \Dp \Dm \dG(u), \Dp \Dm u \rangle &=  \langle \Dp \Dm
  \left(\ave{u}\,Du \right), \Dp \Dm u \rangle \\ 
  &  = \langle \Dm \ave{u}\, \Dp(Du), \Dp \Dm u \rangle +  \langle S^-
  \ave{u}\, \Dp \Dm (Du), \Dp \Dm u \rangle \\ 
  &  \quad+ \langle \Dp \ave{u}\, \Dp(Du), \Dp \Dm u \rangle +
  \langle \Dp \Dm \ave{u}\, Du, \Dp \Dm u \rangle \\ 
  &=: \mathcal{E}^1(u) + \mathcal{E}^2(u) + \mathcal{E}^3(u) +
  \mathcal{E}^4(u),
\end{align*}
in the obvious notation. 
By the discrete
Sobolev inequality (cf.~\cite[Lemma A.1]{HoldenKoleyRisebro:2013})
\begin{equation*}
  \norm{\Dm u}_{\infty}\le \sqrt{\frac32} \left( \norm{\Dp\Dm u}+\norm{u}\right),
\end{equation*}
and the relation $\norm{\Dp\Dm u}= \norm{\Dp^2 u}$,
we apply the Cauchy--Schwarz  inequality to obtain
\begin{align*}
  \abs{\mathcal{E}^1(u)} &\le  \norm{\Dm \ave{u}}_{\infty} \norm{\Dp
    Du} \norm{\Dp \Dm u} \\ 
  &\le  \norm{\Dm \ave{u}}_{\infty}\frac12\big(
  \norm{\Dp^2u}+\norm{\Dp \Dm u} \big)\norm{\Dp \Dm u} \\ 
  &= \norm{\Dm {u}}_{\infty} \norm{\Dp \Dm u}^2 \\ 
  & \le\sqrt{\frac32} (\norm{\Dp \Dm u} + \norm{u}) \norm{\Dp \Dm u}^2 \\ 
  & \le \sqrt{\frac32} \, \norm{\Dp \Dm u}\, \norm{u}_{h^2}^2.
\end{align*}
Similar arguments show that 		
\begin{align*}
\abs{\mathcal{E}^3(u)} \le \sqrt{\frac32}\, \norm{\Dp \Dm u}\, \norm{u}_{h^2}^2, 
\,\text{and}\,\abs{\mathcal{E}^4(u)} \le \sqrt{\frac32} \, \norm{\Dp \Dm u}\, \norm{u}_{h^2}^2.
\end{align*}		
To estimate the last term, we proceed as follows:
\begin{align*}
  \mathcal{E}^2(u) &:= \langle S^- \ave{u}\, \Dp \Dm (Du), \Dp \Dm u
  \rangle\\ 
  &=\langle S^- \ave{u}\, D(\Dp \Dm u), \Dp \Dm u \rangle \\
  &=\langle S^-u\, \Dp\Dm u,D(\Dp\Dm u)\rangle\\
  &= -  \langle D\left(S^- \ave{u}\,\Dp \Dm u  \right), \Dp \Dm u
  \rangle \\ 
  &= -\frac{\Dx}{2} \langle \Dp(S^- \ave{u}) \, D(\Dp \Dm u), \Dp \Dm
  u \rangle \\ 
  &\qquad- \frac12 \langle S^- \Dp \Dm u\,  D(S^- \ave{u}), \Dp \Dm u
  \rangle\quad\text{by \eqref{imp1}}\\ 
  &=: \mathcal{E}^{21}(u)+ \mathcal{E}^{22}(u).
\end{align*}
Again using the discrete Sobolev inequality (cf. \cite[Lemma
A.1]{HoldenKoleyRisebro:2013}) we see that
\begin{align*}
  \abs{\mathcal{E}^{21}(u)} &\le \frac{\Dx}{2} \norm{ \Dp(S^-
    \ave{u})}_{\infty} \, \norm{D\Dp \Dm u} \, \norm{\Dp \Dm u}\\ 
  &=\norm{\Dm u}_{\infty} \left(\Dx  \norm{D\Dp \Dm u}
  \right)\norm{\Dp \Dm u}\\ 
  &\le\norm{ \Dm u}_{\infty} \,  \norm{\Dp \Dm u} \, \norm{\Dp
    \Dm u}\\ 
  & \le \sqrt{\frac32} \, \norm{\Dp \Dm u}\, \norm{u}_{h^2}^2.
\end{align*}
Similarly,
\begin{align*}
\abs{\mathcal{E}^{22}(u)} \le \sqrt{\frac32}\, \norm{\Dp \Dm u}\, \norm{u}_{h^2}^2.
\end{align*}
Therefore, we conclude
\begin{align*}
\abs{\mathcal{E}^2(u)} \le \sqrt{\frac32} \, \norm{\Dp \Dm u}\, \norm{u}_{h^2}^2.
\end{align*}	
Hence
\begin{align*}
  2 \frac{ \langle \Dp \Dm \dG(u^{n+1/2}), \Dp \Dm u^{n+1/2}
    \rangle}{\norm{\Dp \Dm u^{n+1}} + \norm{\Dp \Dm u^n}}
  &\le 2\sqrt{\frac32} \frac{\norm{\Dp \Dm u^{n+1/2}}\,
    \norm{u^{n+1/2}}_{h^2}^2}{\norm{\Dp \Dm u^{n+1}} + \norm{\Dp \Dm
      u^n}} \\ 
  & \le\sqrt{\frac32}\norm{u^{n+1/2}}_{h^2}^2,
\end{align*}
which by \eqref{eq:dtbound} implies that
\begin{equation}\label{eq:a1}
  \abs{D_+^t\left(\norm{\Dp \Dm u^n}\right)} \le  \sqrt{\frac32}\,\norm{u^{n+1/2}}_{h^2}^2. 
\end{equation}
In the same manner, applying the operator $\Dp$ to \eqref{eq:BO_fd},
and subsequently taking the inner product with $\Dp u^{n+1/2}$, yields
\begin{equation*}
  D_+^t\left(\norm{\Dp  u^n}\right) = 2 \frac{ \langle \Dp
    \dG(u^{n+1/2}), \Dp  u^{n+1/2} \rangle}{\norm{\Dp  u^{n+1}} +
    \norm{\Dp  u^n}}. 
\end{equation*}
Using the discrete 
Sobolev inequality $\norm{u}_\infty \le \norm{u}_{h^1}$ 
\begin{align*}
 \abs{\langle \Dp \dG(u), \Dp u \rangle}
 &= \abs{\langle \ave{u}Du, \Dm\Dp u \rangle}\\
 &\le \norm{u}_\infty \,\norm{Du} \,\norm{\Dp\Dm u}\\
 &\le \norm{\Dp u} \norm{u}_{h^2}^2.
\end{align*}
Thus, we obtain
\begin{equation} \label{eq:a2}
  \abs{D_+^t\left(\norm{\Dp u^n}\right)}\le \sqrt{\frac32}\norm{u^{n+1/2}}_{h^2}^2.
\end{equation}
Furthermore, the conservative property \eqref{eq:l2} implies that
\begin{equation} \label{eq:a3}
D_+^t\left(\norm{u^n}\right) =0.
\end{equation}  
Combining \eqref{eq:a1}, \eqref{eq:a2}, and \eqref{eq:a3} concludes the proof. 
\end{proof}

We can now state the following stability result:
\begin{theorem}
  \label{thm2}
  If the initial function $u_0$ is in $H^2$, 
  then there exist
  a time $T>0$ and a constant $C$, both depending only on $\norm{u_0}_{H^2}$, such that 
  \begin{align*}
    \norm{u^n}_{h^2} \le C, \quad \text{for $t_n \le T$}
  \end{align*}
  for all sufficiently small $\lambda=\Dt/\Dx$. 
\end{theorem}
\begin{proof}
  Set $y_n = \norm{u^{n}}_{h^2}$. By Lemma~\ref{lemma1}, we have
  $\norm{u^{n+1/2}}\le K\norm{u^n}$, so that Lemma~\ref{lemma3.1}
  gives
  \begin{equation*}
    y_{n+1}\le y_n + \sqrt{\frac32}\left(K y_n\right)^2
  \end{equation*}
  for all $\Dt/\Dx\le\lambda_n=L/(K \norm{u^n}_{h^2})$.  We choose a time discretization $\Dt_n$.
  Let $w(t)$ solve the differential equation $w'(t)=\sqrt{3/2}K^2 w(t)^2$,
  $w(0)=\norm{u_0}_{H^2}$. This equation has a blow up time $\hat{T}=1/(\sqrt{3/2}K^2 \norm{u_0}_{H^2})$,
  and for $t<T$, $w$ is strictly increasing. Choose $T<\hat{T}$, we
  have that $w(t)\le w(T)$, and we claim that also $y_n\le w(t_n)\le
  w(T)$ for $t_n\le T$. This claim is true for $n=0$, and we
  inductively assume that it is true for $n=0,\ldots,N$. Then
  \begin{align*}
    y_{N+1}=y_N + \Dt_n CK^2 y_N^2 &\le w(t_N) + \int_{t_N}^{t_{N+1}}
     \sqrt{\frac32}K^2 w(t_N)^2\,dt\\
    &\le w(t_N)+\int_{t_N}^{t_{N+1}} w'(s)\,ds = w(t_{N+1}).
  \end{align*}
  This proves that $y_n\le w(T)$ for all $n$ such that $t_n\le T$, thus $ \norm{u^n}_{h^2} \le C=w(T)$.   We can now use a uniform spacing, and let 
  $\Dt/\Dx\le\lambda\le L/(KC)$. 
\end{proof}

Now we turn to the estimate of the temporal derivative of approximate solution $u^n$.
This bound will enable us to apply the Arzel\`a--Ascoli theorem in order to prove the convergence of an approximate solution $u^n$.  From the scheme \eqref{eq:BO_fd}, using the propety $\norm{\Dp\Dm u}=\norm{ \dH (\Dp\Dm u)}$,
 we see that 
\begin{align*}
\norm{D^t_{+} u^n}\le \norm{\dG(u^{n+1/2})} \, + \, \norm{\Dp\Dm u^{n+1/2}}.
\end{align*}
By the discrete Sobolev inequality
$$
\norm{\dG(u^{n+1/2})} \le \norm{u^{n+1/2}}_{\infty} \norm{Du^{n+1/2}} \le C \norm{u^{n+1/2}}_{h^2}^2.
$$
Therefore Theorem~\ref{thm2} implies that $\norm{D^t_{+} u^n}\le C$.

Thus, we can follow Sj\"oberg \cite{Sjoberg:1970} to prove
convergence of the scheme \eqref{eq:BO_fd} for $t<T$. We reason as follows: We
construct the piecewise quadric continuous interpolation
$u_{\Dx}(x,t)$ in two steps. First we make a spatial interpolation for
each $t_n$:
\begin{equation} \label{eq:bilinearinterp}
  \begin{aligned}
    u^n(x) &=u_j^n +(x-x_j)Du_j^n \\
    &\quad +\frac12(x-x_j)^2\Dp\Dm u_j^n, \quad x\in
    [x_j,x_{j+1}), \, j\in\Z.
  \end{aligned}
\end{equation}
Next we interpolate in time: 
\begin{equation} \label{eq:bilinearinterp1} u_{\Dx}(x,t) =u^n(x)
  +(t-t_n)\Dtp u^n(x), \quad x\in \R, \, t\in [t_n, t_{n+1}], \,
  (n+1)t_{n+1}\le T.
\end{equation}
Observe that
\begin{equation*}
  u_{\Dx}(x_j,t_n) =u_j^n, \qquad j\in\Z, \quad n\in\N_0.
\end{equation*}
Note that $u_\Dx$ is continuous everywhere and continuously
differentiable in space.

The function $u_\Dx$ satisfies for $x\in [x_j,x_{j+1})$ and $t\in
[t_n, t_{n+1}]$
\begin{align}
  \partial_x u_\Dx(x,t)&=Du^n_j+(x-x_j)\Dp\Dm u^n_j \label{eq:udxH1C}  \\
  &\quad  +(t-t_n)\Dtp\Big( Du^n_j+(x-x_j)\Dp\Dm u^n_j \Big), \notag\\
  \partial_x^2 u_\Dx(x,t)&= \Dp\Dm u^n_j+(t-t_n)\Dtp \Dp\Dm u^n_j,\label{eq:udxtL2C} \\
  \partial_t u_\Dx(x,t)&= \Dtp u^n(x),\label{eq:udxtL2C1}
\end{align}
which implies
\begin{align}
  \norm{u_{\Dx}(\dott,t)}_{L^2(\R)}&\le \norm{u_0}_{L^2(\R)},\label{eq:udxL2}\\
  \norm{\partial_x u_\Dx(\dott,t)}_{L^2(\R)} &\le C,\label{eq:udxH1}\\
  \norm{\partial_t u_\Dx(\dott,t)}_{L^2(\R)} &\le C,\label{eq:udxtL2} \\
  \norm{\partial_{x}^2 u_\Dx(\dott,t)}_{L^2(\R)}&\le
  C,\label{eq:udxH3}
\end{align}
for $t\le T$ and for a constant $C$ which is independent of
$\Dx$. 
The bound on $\partial_t u_\Dx$ also implies that $u_\Dx\in
\Lip([0,T];L^2(\R))$. Then an application of the
Arzel\`a--Ascoli theorem using \eqref{eq:udxL2} shows that the set
$\seq{u_\Dx}_{\Dx>0}$ is sequentially compact in
$C([0,T];L^2(\R))$. Thus there exists a sequence
$\seq{u_{\Dx_j}}_{j\in\N}$ which converges uniformly in
$C([0,T];L^2(\R))$ to some function $u$. 

Next we show that the limit $u$ is a weak solution of the Cauchy
problem \eqref{eq:main}, i.e., $u$ satisfies
\begin{equation}
  \label{weak}
    \int_0^T \int_{-\infty}^{\infty} \big(u \psi_t - \frac{u^2}{2}\psi_x - 
    u H(\psi_{xx})\big) \,dxdt 
    +\int_{-\infty}^{\infty} \psi(x,0)u_0(x)\,dx=0,
\end{equation}
for all test functions $\psi\in C^\infty_0(\R\times[0,T))$.

To do this, we start by noting that the piecewise constant function
\begin{equation*}
  \bar{u}_{\Dx}(x,t)=u^n_j\quad \text{ for $(x,t)\in
    [x_j,x_{j+1})\times [t_n,t_{n+1})$,}
\end{equation*}
also converges to $u$ in $L^\infty([0,T];L^2_{\mathrm{loc}}(\R))$. It is more
convenient to apply a Lax--Wendroff type argument to $\bar{u}_\Dx$
than to $u_\Dx$. 

Let $\psi \in C_0^{\infty}(\R \times [0,T))$ be any test function and denote $\psi_j^n = \psi(x_j, t_n)$.
Multiplying the scheme \eqref{eq:BO_fd} by $\Dx \Dt \psi_j^n$, and subsequently summing over all $j$ and $n$ yields 
\begin{align*}
  \Dx \Dt \sum_{j} \sum_{n} \psi_j^n \,D^t_{+} u^n_j & = \Dx \Dt
  \sum_{j} \sum_{n} \psi_j^n  \,\dG(u^{n+1/2})_j \\
  & \quad - \Dx \Dt \sum_{j} \sum_{n} \psi_j^n\, \dH(\Dp \Dm
  u^{n+1/2})_j.
\end{align*}
It is straightforward to show that
\begin{align*}
\Dx \Dt \sum_{j} \sum_{n} \psi_j^n \,D^t_{+} u^n_j  
&= - \Dx \Dt \sum_{j} \sum_{n} u^n_j\, D^t_{-} \psi_j^n  - \Dx  \sum_{j}  \psi_j^0 \, u^0_j \\
&  \to -\int_{\R} \int_0^T u \psi_t \,dx\,dt - \int_{\R} \psi(x,0) u_0(x) \,dx \, \text{as} \,\, \Dx \downarrow 0.
\end{align*}
Next, for the nonlinear term, we proceed as follows:
\begin{align*}
  \Dx \Dt \sum_{j} \sum_{n} \psi_j^n \,\dG(u^{n+1/2})_j
  &= \Dx \Dt \sum_{j} \sum_{n} \psi_j^n \, \widetilde{u_j^{n+1/2}}\, Du_j^{n+1/2} \\
  &= \Dx \Dt \sum_{j} \sum_{n} \psi_j^n \Bigg[ \frac13
  D\Big(u_j^{n+\frac12}\Big)^2 + \frac13 u_j^{n+1/2} \,D u_j^{n+1/2}
  \Bigg].
\end{align*}
A simple summation-by-parts formula yields
\begin{align*}
\frac13 \Dx \Dt \sum_{j} \sum_{n} \psi_j^n \, D\Big(u_j^{n+1/2}\Big)^2 
&= - \frac13 \Dx \Dt \sum_{j} \sum_{n}  \Big(u_j^{n+1/2}\Big)^2 D\psi_j^n \\
&\to -\frac13 \int_{\R} \int_0^T u^2\, \psi_x \,dx\,dt,  \,\, \text{as} \, \,\Dx \downarrow 0.
\end{align*}
Again, using  summation-by-parts 
\begin{align*}
\frac13 \Dx \Dt \sum_{j} \sum_{n} \psi_j^n  u_j^{n+1/2} \,D u_j^{n+1/2} 
&= - \frac{1}{12} \Dx \Dt \sum_{j} \sum_{n} u_j^{n+1/2}\,u_{j-1}^{n+1/2}\, \Dm \psi_j^n  \\
&\quad - \frac{1}{12} \Dx \Dt \sum_{j} \sum_{n} u_j^{n+1/2}\, u_{j+1}^{n+1/2}\, \Dp \psi_j^n \\
& \to -\frac16 \int_{\R} \int_0^T u^2\, \psi_x \,dx\,dt  \,\, \text{as} \, \,\Dx \downarrow 0.
 \end{align*}
 Here we have used the general formula
 \begin{equation*}
 \langle p , q D q\rangle    = - \frac14\langle q \Sm q ,    \Dm p\rangle  -  \frac14\langle q \Sp q,  \Dp p\rangle.
\end{equation*}
Hence, we conclude 
\begin{equation*}
\Dx \Dt \sum_{j} \sum_{n} \psi_j^n \,\dG(u_j^{n+1/2}) 
\to -\frac12 \int_{\R} \int_0^T u^2\, \psi_x \,dx\,dt  \,\, \text{as} \, \,\Dx \downarrow 0.
\end{equation*}
We are left with the term involving the Hilbert transform. With a
slight abuse of notation we identify a sequence $\seq{v_j}$ with a
piecewise constant function, and use the notation $\langle
\dott,\dott\rangle$ for the $\ell^2$ inner product as well as for the
inner product in $L^2(\R)$. Then
\begin{equation*}
  - \Dx \Dt \sum_{j} \sum_{n} \psi_j^n\, \dH(\Dp \Dm
  u^{n+1/2})_j = \Dt \sum_n \langle u^{n+1/2}, \dH(\Dp\Dm \psi^n)\rangle.
\end{equation*}
Next,
\begin{align*}
  \abs{ \langle u^{n+1/2}, \dH(\Dp\Dm \psi^n) \rangle-\langle
    u,H(\psi_{xx}(\dott,t_n)\rangle}&\le \abs{ \langle u^{n+1/2}-u, \dH(\Dp\Dm
    \psi^n)\rangle} \\
  &\qquad +
  \abs{\langle u,\dH(\Dp\Dm\psi^n)-H(\psi_{xx}(\dott,t_n))\rangle}\\
  &\le \norm{u^{n+1/2}-u}\,\norm{\Dp\Dm \psi^n} \\
  &\qquad +
  \norm{u}\,\norm{\dH(\Dp\Dm\psi^n)-H(\psi_{xx}(\dott,t_n))}.
\end{align*}
The first term on the right will tend to zero, since $u^{n+1/2}$
converges to $u$ in $L^2$. Regarding the second term we have that
the piecewise constant function $\Dp\Dm \psi^n$ will converge to
$\psi_{xx}(\dott,t_n)$ since $\psi$ is smooth, as will the piecewise constant
function $v^n_j:=\psi_{xx}(x_j,t_n)$. Using these observations
\begin{align*}
  \norm{\dH(\Dp\Dm\psi^n)-H(\psi_{xx}(\dott,t_n))} &\le 
  \norm{\dH(\Dp\Dm\psi^n-v^n)}+\norm{\dH(v^n)-H(\psi_{xx}(\dott,t_n))}\\
  &\le \norm{\Dp\Dm\psi^n-v^n} + \norm{\dH(v^n)-H(\psi_{xx}(\dott,t_n))}.
\end{align*}
We have already observed that the first term on the right will tend to
zero as $\Dx$ to zero, and the second term will vanish by
Lemma~\ref{lem:l2convergence} since $\psi_{xx}$ is smooth. Thus we
have established that
\begin{equation*}
\Dx \Dt \sum_{j} \sum_{n} \psi_j^n\, \dH(\Dp \Dm
  u^{n+1/2})_j \to
  - \int_0^T \int_{\R} u H(\psi_{xx}) \,dxdt  \,\, \text{as} \, \,\Dx \downarrow 0,
\end{equation*}
which shows that $u$ is a weak solution.

The bounds \eqref{eq:udxH1}, \eqref{eq:udxtL2}, and \eqref{eq:udxH3}
mean that $u$ is actually a strong solution such that \eqref{eq:main}
holds as an $L^2$ identity. Thus the limit $u$ is the unique solution
to the BO equation \eqref{eq:main} taking the initial data $u_0$.

Summing up, we have proved the following theorem:
\begin{theorem}\label{thm:H3convergence}
  Assume that $u_0\in H^2(\R)$.  Then there exists a finite time
  $T$, depending only on $\norm{u_0}_{H^2(\R)}$, such that for
  $t\le T$, the difference approximations defined by
  \eqref{eq:BO_fd} converge uniformly in $C(\R\times [0,T])$ to
  the unique solution of the Benjamin--Ono equation \eqref{eq:main} as
  $\Dx\to 0$ with $\Dt=\order{\Dx}$.
\end{theorem}


\section{The periodic case}\label{sec:periodic}
To keep the presentation fairly short we have only provided details in
the full line case.  However, the same proofs apply also in the
periodic case but the discrete Hilbert transform is defined
differently. In this case it should be an approximation of
the singular integral
\begin{equation}
  \label{PHT}
  \Hil u(x) = \mathrm{P.V.} \frac{1}{2L} \int_{-L}^{L}\!\!
  \cot\left(\frac{\pi}{2L} (x-y)\right) u(y)\,dy, 
\end{equation}
such that Lemma~\ref{lemma:imp} holds.
A simple use of the trigonometric identity
\begin{align*}
  2\cot(\theta)=\cot\left(\frac\theta2\right) -
  \tan\left(\frac\theta2\right),
\end{align*}
helps use to rewrite \eqref{PHT} as
\begin{equation*}
  \Hil u:= T_1  u  - T_2 u,
\end{equation*}
where
\begin{equation}
  T_1u(x) = \mathrm{P.V.} \frac{1}{4L} \int_{-L}^{L}
  \cot\left(\frac{\pi}{4L} (x-y)\right) u(y)\,dy, 
\end{equation}
and
\begin{equation}
  T_2u(x) = \mathrm{P.V.} \frac{1}{4L} \int_{-L}^{L}
  \tan\left(\frac{\pi}{4L} (x-y)\right) u(y)\,dy. 
\end{equation}
Let $n$ be an \emph{even} integer such that $0\le n\le N-1$.  For
this $n$, we have
\begin{align*}
  T_1u(x_n) &= \mathrm{P.V.} \frac{1}{4L} \int_{x_0}^{x_N} \cot\left(\frac{\pi}{4L} (x_n-y)\right) u(y)\,dy\\
  &= \frac{1}{4L}\sum_{j=0}^{\frac{N-3}{2}} \int_{x_{2j}}^{x_{2j+2}} \cot\left(\frac{\pi}{4L} (x_n-y)\right) u(y)\,dy \\
  &\quad + \frac{1}{4L} \int_{x_{N-1}}^{x_N} \cot\left(\frac{\pi}{4L}
    (x_n-y)\right) u(y)\,dy.
\end{align*}
We apply the midpoint rule on each of these integrals in the sum and
endpoint rule for the last integral, and we obtain the following
quadrature formula:
\begin{equation}
  \label{T_1quad}
  \begin{aligned}
    T_1u(x_n)&=\frac{1}{4L}\sum_{j=\,\text{odd}}2\Dx \, \,u(x_j)\cot\left(\frac{\pi}{4L} (x_n-x_j)\right)\\
    &\quad +\frac{1}{4L}\Dx \, \,u(x_N)\cot\left(\frac{\pi}{4L}
      (x_n-x_N)\right).
  \end{aligned}
\end{equation}
Using the identity $\Dx=2L/N$, we define
\begin{equation}
  \label{T_1quada}
  T_1u_n=\frac{1}{N}\sum_{j=\,\text{odd}}u_j\cot\left(\frac{\pi(n-j)}{2N} \right)\\
  +\frac{1}{2N}u(x_N)\cot\left(\frac{\pi}{4L} (x_n-x_N)\right).
\end{equation}
Next we write $T_2u(x_n)$ as
\begin{align*}
  T_2u(x_n) &= \mathrm{P.V.} \frac{1}{4L} \int_{x_0}^{x_N}
  \tan\left(\frac{\pi}{4L} (x_n-y)\right) u(y)\,dy\\ 
  & = \frac{1}{4L}\sum_{j=\,\text{odd}}\int_{x_{j}}^{x_{j+2}}
  \tan\left(\frac{\pi}{4L} (x_n-y)\right) u(y)\,dy \\ 
  &\quad + \frac{1}{4L} \int_{x_{0}}^{x_1} \tan\left(\frac{\pi}{4L}
    (x_n-y)\right) u(y)\,dy.
\end{align*}
To obtain the quadrature formula, we use the midpoint rule on each of
the integral in the sum and endpoint rule on the last integral,
\begin{equation}
  \label{T_2quad}
  \begin{aligned}
    T_2u(x_n)&=\frac{1}{4L}\sum_{j=\,\text{even},j\neq 0}2\Dx \,
    \,u(x_j)\tan\left(\frac{\pi}{4L} (x_n-x_j)\right)\\ 
    &\quad  +\frac{1}{4L}\Dx \, \,u(x_0)\tan\left(\frac{\pi}{4L}
      (x_n-x_0)\right).
  \end{aligned}
\end{equation}
Using the identity $\Dx=2L/N$, we have
\begin{equation}
  \label{T_2quada}
  T_2u_n=\frac{1}{N}\sum_{  j=\,\text{even},\;   j\neq 0  }     u_j
  \tan   \left(\frac{\pi(n-j)}{2N} \right)\\ 
  +\frac{1}{2N}      u(x_0)         \tan     \left(   \frac{\pi}{4L}
    (x_n-x_0)   \right). 
\end{equation}
Since $u$ is $N$-periodic grid function, we have
\begin{align*}
  u(x_N)\cot\left(\frac{\pi}{4L} (x_n-x_N)\right)= -
  \,u(x_0)\tan\left(\frac{\pi}{4L} (x_n-x_0)\right).
\end{align*}
Therefore, adding \eqref{T_2quada} and \eqref{T_1quada} we have, for
even $n$
\begin{align*}
  (\mathbb{\Hil}
  u)_n=\frac{1}{N}\sum_{j=\,\text{odd}}u_j\cot\left(\frac{\pi(n-j)}{2N}
  \right) - \frac{1}{N}\sum_{ j=\,\text{even} } u_j \tan
  \left(\frac{\pi(n-j)}{2N} \right).
\end{align*}
Similarly, we have for odd $n$
\begin{align*}
  (\mathbb{\Hil}
  u)_n=\frac{1}{N}\sum_{j=\,\text{even}}u_j\cot\left(\frac{\pi(n-j)}{2N}
  \right) - \frac{1}{N}\sum_{ j=\,\text{odd} } u_j \tan
  \left(\frac{\pi(n-j)}{2N} \right).
\end{align*}
Combining above two relations, we conclude
\begin{equation}
  \label{b formula}
  \mathbb{\Hil} u = c*u,
\end{equation}
where the vector $c$ is given by
\begin{equation}
  \label{c formula}
  c_n=\frac{1-(-1)^n}{2N}  \cot\left( \frac{\pi n}{2N}\right) -
  \frac{1+(-1)^n}{2N}  \tan\left( \frac{\pi n}{2N}\right). 
\end{equation}
Next we prove the following properties of discrete Hilbert transform
$\mathbb{\Hil}$ defined by \eqref{b formula}--\eqref{c formula}:
\begin{lemma}
\label{lemma123}
  The discrete Hilbert transform is skew symmetric. Moreover, it 
  satisfies $\norm{\mathbb{\Hil} u}\le\norm{u}$ and $\norm{u}=\norm{\mathbb{\Hil}
    u}$ provided $\sum_{j=0}^{N-1}u_{j}=0$. Furthermore, we have
  \begin{align*}
    \norm{\mathbb{\Hil}\Dp\Dm u}=\norm{\Dp\Dm u}.
  \end{align*}
\end{lemma}
\begin{proof}
  The skew-symmetric property of $\mathbb{\Hil}$ follows from the fact
  that $c_{-n}=-c_n$, for any $n$.  Furthermore, we use the
  \emph{discrete Fourier transform} (DFT) to prove that
  $\mathbb{\Hil}$ preserves the $\ell^2$-norm.

  First we recall the definition of discrete Fourier transform.  For a
  given $N$-periodic grid function $u$, we define the DFT by
  \begin{align*}
    \hat{u}_k= \sum_{n=0}^{N-1}u_n\, e^{-i\frac{2\pi k n}{N}}, \quad
    k=0,1,2,..., N-1,
  \end{align*}
  and the inversion formula is then
  \begin{align*}
    u_k=\frac1N \sum_{n=0}^{N-1}\hat{u}_n\, e^{i\frac{2\pi k n}{N}},
    \quad k=0,1,2,..., N-1.
  \end{align*}
  Then the Parseval formula reads
  \begin{align*}
    \norm{\hat{u}}= \sqrt{N}\norm{u}.
  \end{align*}

  Next we compute the DFT of $c$. We claim that the Fourier transform
  of $c$ is given by
  \begin{equation}
    \label{chat}
    \hat{c}_n=
    \begin{cases}
      -i & \text{for $n=1,2,...,\frac{N-1}{2}$},\\
      \,0 &\text{for $n=0$},\\
      \,i &\text{for $n=\frac{N+1}2,....., N-2,N-1$}.
    \end{cases}
  \end{equation}
  To prove this we use inverse discrete Fourier transform. From
  \eqref{chat}, we see that
  \begin{align*}
    \sum_{k=0}^{N-1} \hat{c}_k e^{i\frac{2\pi k n}{N}}&=
    -i\sum_{k=1}^{(N-1)/2} \hat{c}_k e^{i\frac{2\pi k n}{N}}
    +i\sum_{k=(N+1)/2}^{N-1} \hat{c}_k e^{i\frac{2\pi k n}{N}}\\
    &=2\sum_{k=1}^{(N-1)/2}\sin\left(\frac{2\pi 
         kn}{N}\right)\\ 
    &= 2 \, \Im\left( \sum_{k=1}^{(N-1)/2}\exp\left(\frac{2\pi ikn}{N}\right) \right)\\
    &= 2\, \Im\left( \frac{e^{i\frac{2\pi
            n}{N}\frac{N-1}{2}}-1}{e^{i\frac{2\pi  n}{N}}-1}
      e^{i\frac{2\pi  n}{N}}\right)\\ 
    &=  - \Im\left(i \frac{e^{i\frac{2\pi
            n}{N}\frac{N-1}{2}}-1}{\sin(\frac{\pi  n}{N})}
      e^{i\frac{\pi  n}{N}}\right)\\ 
    &=  - \Im\left(i \frac{(-1)^n e^{-i\frac{\pi
            n}{N}}-1}{\sin(\frac{\pi  n}{N})}  e^{i\frac{\pi  n}{N}}\right)\\
    &=  - \Im\left(i \frac{(-1)^n -e^{i\frac{\pi  n}{N}}}{\sin(\frac{\pi  n}{N})} \right)\\
    &=   \cot \left(\frac{\pi n}{N}\right) - \frac{(-1)^n}{\sin(\pi n/N)} \\
    &= \frac{\cos^2(\frac{\pi n}{2N})- \sin^2(\frac{\pi
        n}{2N})}{2\sin(\frac{\pi n}{2N})\cos(\frac{\pi n}{2N})}
    - \frac{(-1)^n\big(\cos^2(\frac{\pi n}{2N})+ \sin^2(\frac{\pi
        n}{2N}) \big)}{2\sin(\frac{\pi n}{2N})\cos(\frac{\pi n}{2N})}
    \\ 
    &=N\Big(\frac{1-(-1)^n}{2N} \cot\left( \frac{\pi n}{2N}\right) -
    \frac{1+(-1)^n}{2N}  \tan\left( \frac{\pi n}{2N}\right)\Big)\\
    &=N c_n.
  \end{align*}
  This proves the claim.  Therefore, we have
  \begin{align*}
    \widehat{\mathbb{\Hil} u}_n=\hat{c}_n \, \hat{u}_n.
  \end{align*}
  Now using Parseval's formula we have
  \begin{align*}
    \norm{\mathbb{\Hil}u}&= \norm{\widehat{\mathbb{\Hil} u}}\\
    &=\norm{\hat{c}\, \hat{u}}\\
    &= \Big(\sum_{n=1}^{N-1} \abs{\hat{u}}^2_n\Big)^{1/2} \\
    &\le \norm{\hat{u}}\\
    &=\norm{u}.
  \end{align*}
  Thus we have $ \norm{\mathbb{\Hil}u}\le \norm{u}$, and $
  \norm{\mathbb{\Hil}u}= \norm{u}$ provided $\hat{u}(0)=0$, that is,
  \begin{align*}
    \sum_{j=0}^{N-1}u_j=0.
  \end{align*}
\end{proof}
Keeping in mind the above discretization for the Hilbert transform, we propose the following implicit scheme to generate
approximate solutions to the BO equation
\eqref{eq:main_per}
\begin{equation}
  \label{eq:BO_fd_per}
  u^{n+1}_j = u^n_j + \Dt\, \dG(u^{n+1/2})_j + \Dt\, \mathbb{\Hil}(\Dp \Dm u^{n+1/2})_j, 
\end{equation}
for $n\ge 0$ and $j=0,\ldots,N-1$.
Regarding $u^0$ we set 
\begin{equation*}
  u^0_j=u_0(x_j),\qquad j= 0, \dots, N-1.
\end{equation*}
Using the properties of the discrete Hilbert transform \eqref{b
formula}--\eqref{c formula}, and using identical arguments to those
used 
in the proof of Theorem~\ref{thm:H3convergence}, we can proove  the following
theorem:
\begin{theorem}
  \label{thm:H3convergence_per}
  Assume that $u_0\in H^2(\mathbb{T})$.  Then there exists a finite
  time $T$, depending only on $\norm{u_0}_{H^2(\mathbb{T})}$,
  such that for $t\le T$, the difference approximations defined
  by \eqref{eq:BO_fd_per} converge uniformly in $C(\mathbb{T}\times
  [0,T])$ to the unique solution of the Benjamin--Ono equation
  \eqref{eq:main_per} as $\Dx\to 0$ with $\Dt=\order{\Dx}$.
\end{theorem}


\section{Numerical experiments}
\label{sec:numerics}
The fully-discrete scheme given by \eqref{eq:BO_fd} has been tested on
suitable test cases, namely soliton interactions, in order to
demonstrate its effectiveness.  It is well-known that a soliton is a
self-reinforcing solitary wave that maintains its shape while
traveling at constant speed. Solitons are the result of a delicate
cancellation of nonlinear and dispersive effects in the
medium. Several authors, see, e.g.,
\cite{BoydXu,vasu:1998,Pelloni:2000} have studied the soliton
interactions for the BO equation.
\subsection*{A one-soliton solution}
The Benjamin--Ono equation
\eqref{eq:main_per} has one-periodic wave solution that tend towards
the one-soliton in the long wave limit, i.e., when the wave number
goes to zero.  It is given by
\begin{equation}
  \begin{aligned}
    u(x,t) = -\frac{2c\delta^2}{1 -\sqrt{1-\delta^2} \cos(c\delta(x
      -ct))}, \quad \text{with} \quad \delta=\frac{\pi}{cL},
  \end{aligned}
  \label{eq:BO_onesol}
\end{equation}
where $L$ denotes the period and $c$ is the wave speed.

We have applied scheme \eqref{eq:BO_fd_per} to simulate the periodic
one wave solution \eqref{eq:BO_onesol} with $L=15$, $c=0.25$ and
initial data $u_0(x) =u(x,0)$. The exact solution is periodic in time
with the period $p=120$. In Figure~\ref{fig:1} we show the approximate
and exact solution at $t=4p=480$.
\begin{figure}[h]
  \centering
  \includegraphics[width=0.7\linewidth]{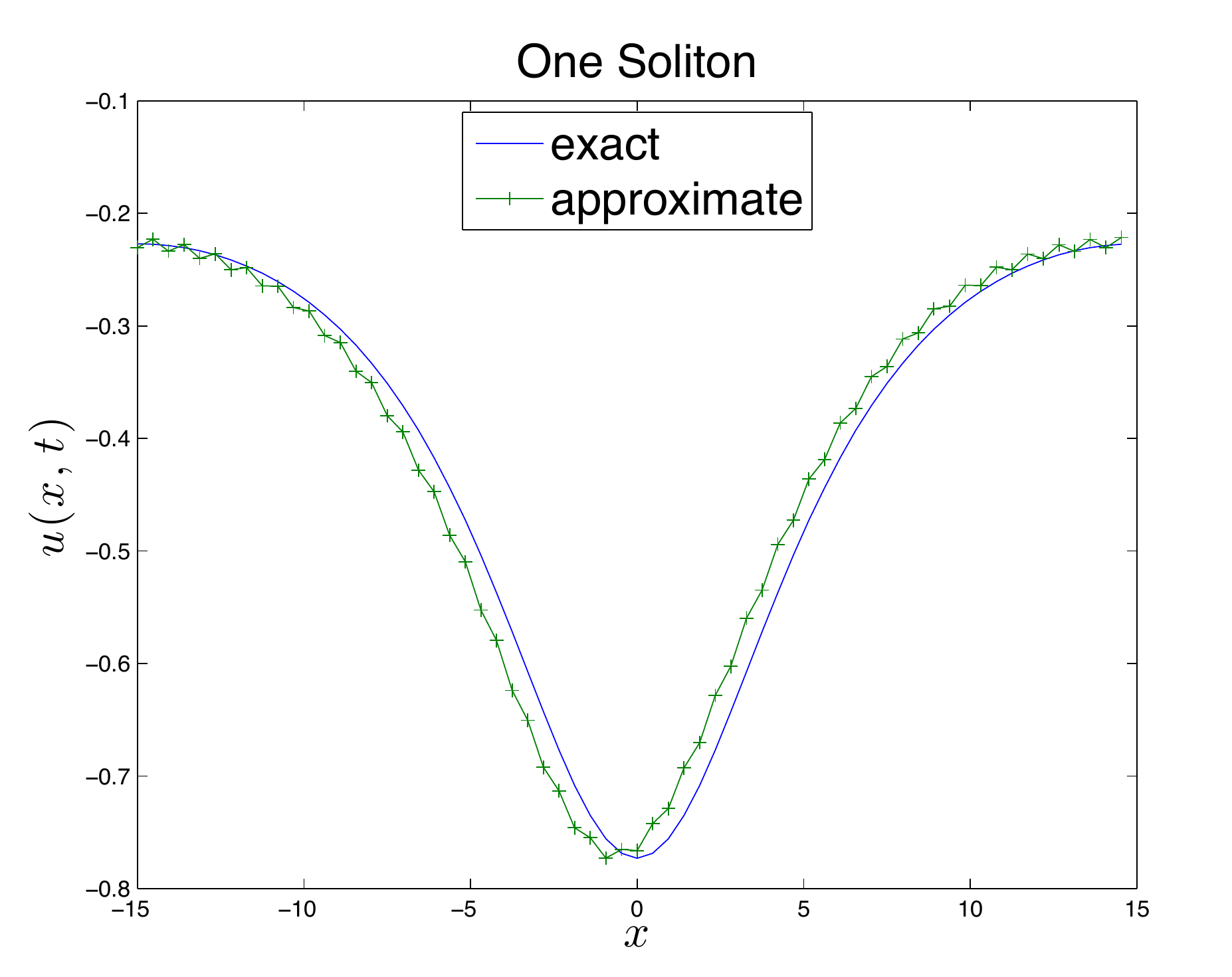}
  \caption{Comparison of exact and numerical solutions with initial
    data \eqref{eq:BO_onesol}. }
  \label{fig:1}
\end{figure}
We have also computed numerically the error for a range of $\Dx$,
where the relative $L^2$ error at time $T$ is defined by
\begin{equation*}
  E1(T)=100 \frac{\norm{u-u_{\Dx}}_2}{\norm{u}_{2}}
\end{equation*}
where the norms were computed using the trapezoid rule on the points $x_j$,
and the relative $L^{\infty}$ error is defined by
\begin{equation*}
  E2(T)=100 \frac{\norm{u-u_{\Dx}}_{\infty}}{\norm{u}_{\infty}}.
\end{equation*}
In Table~\ref{tab:1}, we show $L^2$ relative errors as
well as $L^{\infty}$ relative errors for this example at time $T=480$.
\begin{table}[h]
  \centering
  \begin{tabular}[h]{c|r r r r }
    $N$ &\multicolumn{1}{c}{$E1$} &\multicolumn{1}{c}{rate}
    &\multicolumn{1}{c}{$E2$} &\multicolumn{1}{c}{rate}  \\ 
    \hline\\[-2ex]
    33  & 21.24       & &
    23.35           &                                       \\[-1ex]
    65   & 5.76      &\raisebox{1.5ex}{1.9} &    6.75
    &\raisebox{1.5ex}{1.8}   \\[-1ex] 
    129   & 1.46     & \raisebox{1.5ex}{2.0}&    1.71
    &\raisebox{1.5ex}{2.0}      \\[-1ex] 
    257   & 0.39     &\raisebox{1.5ex}{1.9} &    0.49
    &\raisebox{1.5ex}{1.8}     \\[-1ex] 
    513  & 9.75e{-2}     & \raisebox{1.5ex}{2.0}&    1.21e{-1}
    &\raisebox{1.5ex}{2.0}    \\[-1ex] 
    1025  & 3.34e{-2}     & \raisebox{1.5ex}{1.5}&   4.70e{-2}
    &\raisebox{1.5ex}{1.4}      \\[-1ex] 
    2049  & 7.50e{-3}  &\raisebox{1.5ex}{2.1}&      1.07e{-2}     &\raisebox{1.5ex}{2.1}  
  \end{tabular}
  \vspace{1.5ex}
  \caption{$E1$ and $E2$ for the one-soliton solution at time $T=480$.}
  \label{tab:1}
\end{table}
The computed solution in Figure~\ref{fig:1} looks quite well and the errors are also quite low and the convergence rate seems to converge to 2.

\subsection{A two-soliton solution}
\label{sec:twosol}
The velocity of a soliton depends on its amplitude; the higher the
amplitude, the faster it moves.  Thus a fast soliton will overtake a
slower soliton moving in the same direction. After the interaction,
the solitons will reappear with the same shape, but possibly with a
change in phase. As explicit formulas are available, they provide
excellent test cases for numerical methods.

Inspired by \cite{vasu:1998} we use the exact solution
\begin{equation}
  w(x,t) =- \frac{4 c_1 c_2 \left( c_1 \lambda^2_1 + c_2 \lambda^2_2
      + (c_1 + c_2)^3 c_1^{-1} c_2^{-1} (c_1 -c_2)^{-2}
    \right)}{\left( c_1 c_2 \lambda_1 \lambda_2 -(c_1 + c_2)^2 (c_1
      -c_2)^{-2} \right)^2 + (c_1 \lambda_1 + c_2 \lambda_2)^2},
  \label{eq:BO_twosol}
\end{equation}
where $\lambda_j = \lambda_j(x,t) = x -c_j t $, $j=1,2$, and
$c_1, c_2$ are arbitrary constants. Explicit periodic
two-soliton solutions exist, but the exact formula is
complicated. See, e.g., \cite{satsuma} for a more detailed
discussion. In what follows, we have computed the
two-soliton solution \eqref{eq:BO_twosol} of the unrestricted Cauchy
problem \eqref{eq:main_per}.  Moreover, we have used the initial value
$u_0(x) = w(x,-10)$ on an interval $(-30,30)$ as initial
values. and $c_1=2$, $c_1=1$. Since we compute on a finite line we
have used the periodic continuation, and used the scheme for the
periodic case. Since $w(\pm 30,t)$ remains very small in the time
interval $[-10,10]$ we believe that the computed solution is very close
to $w(t-10,x)$ for $t\le 20$, and we use $w(10,x)$ as a reference solution.

Computationally, this is a much harder problem than the one-soliton
solution due to the fact that in this case the errors stem from both
the approximation of the unrestricted initial-value problem by a
periodic one, and by the numerical approximation of the latter. In
Figure~\ref{fig:2} we show the exact solution and the approximate
solutions at $t=20$ computed using $257$ and $513$ grid points in the
interval $[-30,30]$.
\begin{figure}[h]
  \centering
  \includegraphics[width=.7\linewidth]{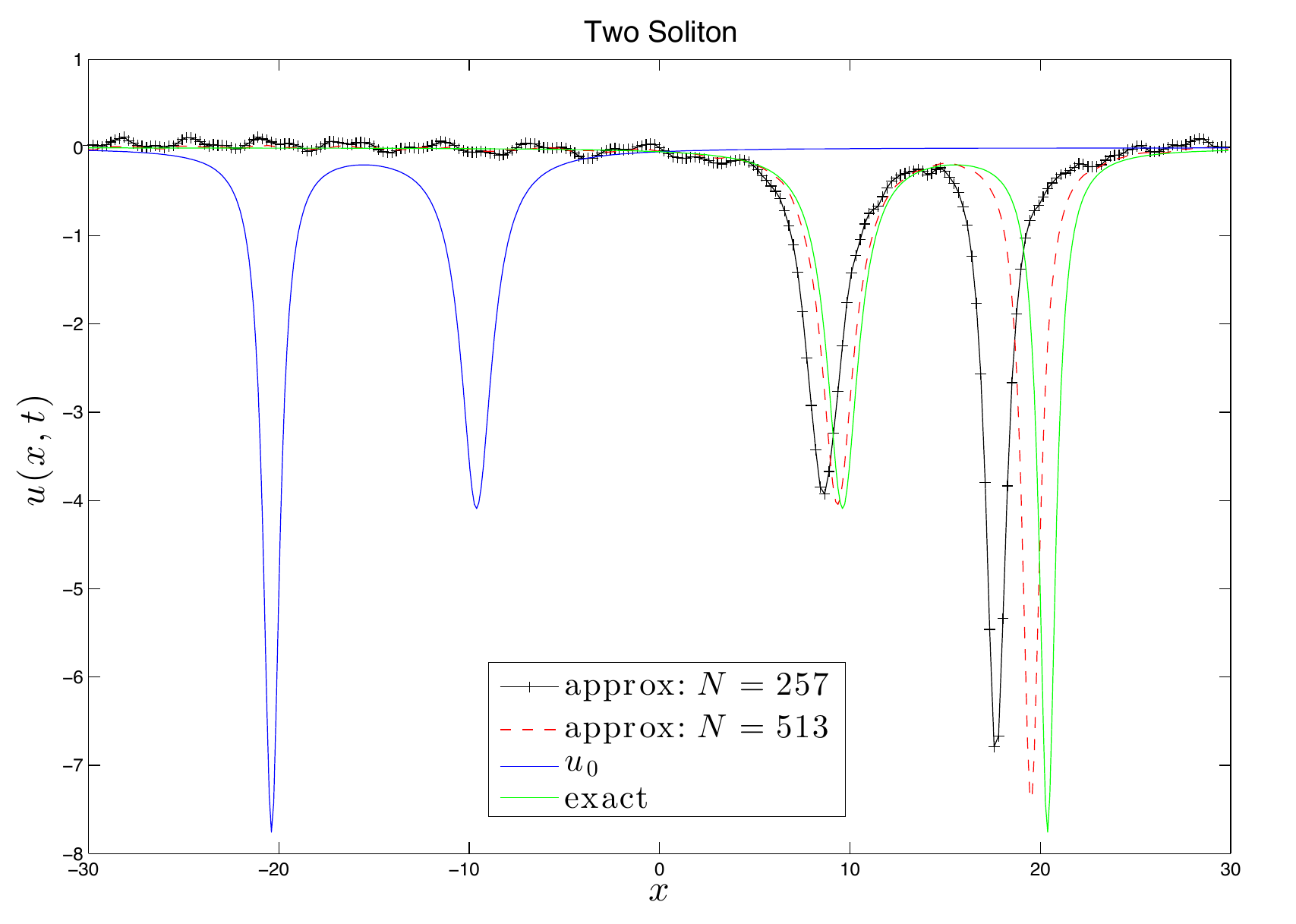}
  \caption{The numerical solution $u_\Dx(x,20)$ with initial data
    $w(x,-10)$.}
  \label{fig:2}
\end{figure}
As the Figure~\ref{fig:2} exhibits, the scheme performs well in the
sense that after the interaction, the two soliton have the same shapes
and velocities as before the interaction.  In Table~\ref{tab:2}, we
show the relative errors $E1$ and $E2$ as well as numerical rate of
convergence for the computed solutions. The large errors and the slow
convergence rate both indicate that we are not yet in asymptotic
regime.
 \begin{table}[h]
   \centering
   \begin{tabular}[h]{c|r r r r }
     $N$ &\multicolumn{1}{c}{$E1$} &\multicolumn{1}{c}{rate}
     &\multicolumn{1}{c}{$E2$} &\multicolumn{1}{c}{rate}  \\ 
     \hline\\[-2ex]
     65   & 125.12       & &
     113.07           &                                       \\[-1ex] 
     129   & 124.76     &\raisebox{1.5ex}{0.0} &    97.26
     &\raisebox{1.5ex}{0.2}   \\[-1ex] 
     257   & 108.74     & \raisebox{1.5ex}{0.2}&    93.99
     &\raisebox{1.5ex}{0.0}      \\[-1ex] 
     513   & 71.34     &\raisebox{1.5ex}{0.6} &    71.20
     &\raisebox{1.5ex}{0.4}     \\[-1ex] 
     1025  & 25.28     & \raisebox{1.5ex}{1.5}&    29.20
     &\raisebox{1.5ex}{1.3}    \\[-1ex] 
     2049  & 6.87     & \raisebox{1.5ex}{1.9}&    7.98
     &\raisebox{1.5ex}{1.9}      \\[-1ex] 
     4097  & 2.16  &\raisebox{1.5ex}{1.7}&      2.52     &\raisebox{1.5ex}{1.7}  
   \end{tabular}
   \vspace{1.5ex}
   \caption{$E1$ and $E2$ for the two-soliton solution at time $T=20$
     with initial data $w(-10,x)$.} 
   \label{tab:2}
 \end{table}



To sum up, our conservative scheme performs very well 
in practice and \emph{proven} to converge, whereas
to the best of our knowledge, there is no constructive proof of convergence,
for the other schemes associated to \eqref{eq:main} or \eqref{eq:main_per},
except \cite{vasu:1998} for some partial result (existence of solution has been assumed) 
in the periodic case \eqref{eq:main_per}.

\end{document}